\documentclass{birkjour}
 \newtheorem{thm}{Theorem}[section]
 
 \newtheorem{lem}[thm]{Lemma}
 \newtheorem{prop}[thm]{Proposition}
 \theoremstyle{definition}
 \newtheorem{defn}[thm]{Definition}
 \theoremstyle{remark}
 \newtheorem{rem}[thm]{Remark}
 
 \numberwithin{equation}{section}

\begin{document}

%-------------------------------------------------------------------------
% editorial commands: to be inserted by the editorial office
%
%\firstpage{1} \volume{228} \Copyrightyear{2004} \DOI{003-0001}
%\seriesextraline{This is the Concrete Title of this Book\br H.E. R and S.T.C. W, Eds.}
%\firstpage{1}

%----------Author 1
\author[Youssef Aserrar and Elhoucien Elqorachi]{Youssef Aserrar and Elhoucien Elqorachi}

\address{% 
	Ibn Zohr University, Faculty of sciences, 
Department of mathematics,\\
Agadir,
Morocco}

\email{youssefaserrar05@gmail.com, elqorachi@hotmail.com }


\date{January 1, 2004}
\title[A d'Alembert type functional equation on  semigroups]{A d'Alembert type functional equation on semigroups}
 
\begin{abstract}
We treat two related trigonometric functional equations on
semigroups. First we solve the $\mu$-sine subtraction law
\[\mu(y) k(x \sigma(y))=k(x) l(y)-k(y) l(x), \quad x, y \in S,\]
for $k, l : S\rightarrow \mathbb{C}$, where $S$ is a semigroup  and 
 $\sigma$ an involutive automorphism, $\mu :S\rightarrow \mathbb{C}$ is a multiplicative function  such that
$\mu (x\sigma (x))=1$ for all $x\in S$, then we determine the complex-valued solutions of the
following functional equation 
\[f(xy) - \mu (y)f(\sigma (y)x) = g(x)h(y),\quad x,y\in S,\] 
on a larger class of semigroups.
\end{abstract}

%%% ----------------------------------------------------------------------
\maketitle

%%% ----------------------------------------------------------------------
%\tableofcontents
\section{Introduction}
Ebanks and Stetkaer \cite{ES} solved the functional equation
\begin{equation}
f(xy) -f(\sigma (y)x) = g(x)h(y),\quad x,y\in M,
\label{EB1}
\end{equation} 
for $f,g,h: M \rightarrow \mathbb{C}$, where $M$ is a group or a monoid generated by its squares and $\sigma :M\rightarrow M$ is an involutive automorphism. That is $\sigma (xy)=\sigma (x)\sigma (y)$ and $\sigma(\sigma(x))=x$ for all $x,y\in M$.\newline
In \cite{BE},  Bouikhalene and Elqorachi determined the complex-valued solutions of the functional equation
\begin{equation}
f(x y)-\mu(y) f(\sigma(y) x)=g(x) h(y), \quad x, y \in M,
\label{E1}
\end{equation}
on groups and monoid generated by its squares, where $\mu :M\rightarrow \mathbb{C}$ is a multiplicative function such that $\mu (x\sigma (x))=1$ for all $x\in M$.
Moreover, the solutions of \eqref{E1} on a semigroup generated by its squares are also known (See Ajebbar and Elqorachi \cite{Ajb}). Recently,  Ebanks \cite{EB3} solved \eqref{EB1} on a larger class of monoids that contains the class of monoids generated by their squares, and regular monoids.\\
 There are some results about solutions of Equation \eqref{EB1} on abelian groups in the literature. The case of $G=\mathbb{R}$ is treated by Kannappan \cite[Equation (vs) in section 3.4.9]{K}. Stetkaer \cite[Corollary 3.5]{ST} derives the solution formulas for the functional equation \eqref{EB1} on abelian groups. The method used in \cite{Ajb,BE,EB3,ES} proceeds from the starting point of the trigonometric functional equation
\begin{equation}
\mu(y) k(x \sigma(y))=k(x) l(y)-k(y) l(x), \quad x, y \in S.
\label{E2}
\end{equation}
The solutions of \eqref{E2} (with $\mu=1$) on a general monoid was given recently by Ebanks \cite[Theorem 4.2]{ES1} in terms of multiplicative and additive functions. In the present paper, we treat the functional equation \eqref{E1} on a general class of semigroups.\\
The contributions of the present work to the knowledge about solutions of \eqref{E1} are the following :
\begin{enumerate}
\item[(1)] The setting has $S$ to be a semigroup, not necessarily generated by its squares, or a monoid.
\item[(2)] We solve Equation \eqref{E2} on semigroups.
\item[(3)] We relate the solutions of \eqref{E1} to the functional equation \eqref{E2} on semigroups, and find explicit formulas for the solutions on a larger class of semigroups.
\end{enumerate}

The outline of the paper is as follows : In the next section we give notations ans terminology. In the third section we solve the functional equation \eqref{E2} on semigroups, and we obtain the explicit formulas for the solutions of Equation \eqref{E1} on a larger class of semigroups $S$.

\section{Notations and Terminology}
Throughout this paper $S$ denotes a semigroup.\\
A function $A: S\rightarrow \mathbb{C}$ is additive if $A(xy) = A(x) + A(y)$ for all $x, y\in S$.\\
A function $\chi: S\rightarrow \mathbb{C}$ is multiplicative if $\chi(xy) = \chi(x)\chi(y)$ for all $x, y \in S$.\\
A function $f: S\rightarrow \mathbb{C}$ is central if $f(xy) = f(yx)$ for all $x, y\in S$, and $f$ is abelian if $f$ is central and $f(xyz)=f(xzy)$ for all $x,y,z\in S.$\\
If $S$ is a semigroup, we define the nullspace
$$
\mathcal{N}_{\mu}(\sigma, S):=\left\lbrace \theta: S \rightarrow \mathbb{C} \mid \theta(x y)-\mu(y) \theta(\sigma(y) x)=0,\quad x, y \in S\right\rbrace .
$$
For any subset $T\subseteq S$ we define $T^2:=\lbrace xy\  \vert \ x, y\in T\rbrace$.
If $\chi: S \rightarrow \mathbb{C}$ is a multiplicative function and $\chi \neq 0$, we define the sets $$I_{\chi}:=\{x \in S \mid \chi(x)=0\}$$ 
$$P_\chi : =\{p\in I_{\chi}\backslash I_{\chi}^2\  \vert up, pv, upv\in I_{\chi}\backslash I_{\chi}^2\  \text{for all}\  u,v\in S\backslash I_{\chi}\}.$$
 Notice that $I_{\chi}$ is  a subsemigroup of $S$ and a prime ideal of $S$.\par  
 For any function $f: S \rightarrow \mathbb{C}$ we define the function $f^{*}(x)=\mu(x) f(\sigma(x))$,\\ $x \in S$, we call $f^{e}:=\frac{f+f^{*}}{2}$ the even part of $f$ and $f^{\circ}:=\frac{f-f^{*}}{2}$ its odd part. The function $f$ is said to be even if $f=f^{*}$, and $f$ is said to be odd if $f=-f^{*}$. If $g, h: S \rightarrow \mathbb{C}$ are two functions we define the function $$(g \otimes h)(x, y):=g(x) h(y),\ x, y \in S.$$
  For a topological semigroup $S$ let $C(S)$ denote the algebra of continuous functions from $S$ into $\mathbb{C}$.
\section{Main result}
The following Lemma will be used later.
\begin{lem}
Let $k, l : S\rightarrow\mathbb{C}$ be a solution of the functional equation \eqref{E2} with $k \neq 0$. Then 
\begin{enumerate}
\item[(1)] $k(xy)=-k^*(yx)$ for all $x,y\in S$.
\item[(2)] If $k=k^*$, then $k(xyz)=0$ for all $x,y,z\in S$.
\end{enumerate}
\label{lem}
\end{lem}
\begin{proof}
(1) Let $x,y \in S$ be arbitrary. By interchanging $x$ and $y$ in \eqref{E2} we get that $\mu (y) k\left( x\sigma (y)\right) =-\mu (x) k\left( y\sigma (x)\right)$, and if we apply this identity to the pair $\left(x, \sigma (y) \right) $, we get $\mu \left( \sigma (y)\right)  k\left( xy\right)=-\mu (x) k\left( \sigma (yx)\right)$.\\
Multiplying this by $\mu (y)$ and by using that $\mu : S\rightarrow \mathbb{C}$ is a multiplicative function such that $\mu \left(y\sigma (y) \right) =1$ for all $y\in S$, we get $k(xy)=-k^*(yx)$ for all $x,y \in S$.\\
(2) By using (1), we get 
$$k(xyz)=-k^*(zxy)=k(yzx)=-k^*(xyz),\quad \text{for all}\quad x,y,z\in S.$$
So if $k=k^*$, we get that $k(xyz)=-k(xyz)$ for all $x,y,z\in S$. This implies that $k(xyz)=0$. This completes the proof of Lemma \ref{lem}.
\end{proof}
In the following proposition we extend the results obtained by Ajebbar and Elqorachi \cite[Proposition 3.3]{Ajb} on  semigroups generated by their squares and by Ebanks \cite[Theorem 4.2]{ES1} on monoids to general semigroups.
\begin{prop}
 The solutions of the functional equation \eqref{E2} with $k \neq 0$ are the following pairs :
 \begin{enumerate}
 \item[(1)] $k$ is any non-zero function such that $k=0$ on $S^2$ and $l=ck$, where $c \in \mathbb{C}$ is a constant.
\item[(2)] $$k=c_1\dfrac{\chi -\chi ^*}{2}\quad\text{and}\quad l=\dfrac{\chi +\chi ^*}{2}+c_2\dfrac{\chi -\chi ^*}{2},$$
where $\chi  : S\rightarrow \mathbb{C}$ is a multiplicative function such that $\chi ^*\neq \chi$ and $c_1 \in \mathbb{C}\backslash  \lbrace0\rbrace,c_2\in \mathbb{C}$ are constants.
\item[(3)] $$k=\left\{ \begin{matrix}
   \chi A & on & S\backslash {{I}_{\chi }}  \\
   0 & on & {{I}_{\chi }}\backslash {{P}_{\chi }}  \\
   \rho  & on & {{P}_{\chi }}  \\
\end{matrix} \right.\quad\text{and}\quad l=\left\{ \begin{matrix}
   \chi (1+cA) & on & S\backslash {{I}_{\chi }}  \\
   0 & on & {{I}_{\chi }}\backslash {{P}_{\chi }}  \\
   c\rho  & on & {{P}_{\chi }}  \\
\end{matrix} \ ,\right.
$$ 
\end{enumerate}
where $c \in \mathbb{C}$ is a constant, $\chi: S \rightarrow \mathbb{C}$ is a non-zero multiplicative function and $A: S \backslash I_{\chi} \rightarrow \mathbb{C}$ is a non-zero additive function such that $\chi^{*}=\chi$, $A \circ \sigma=-A$, and $\rho: P_{\chi} \rightarrow \mathbb{C}$ is the restriction of $k$ to $P_\chi$ such that $\rho^*=-\rho$. In addition we have the following conditions: \\
(I): If $x\in\left\lbrace up,pv,upv \right\rbrace $ for $p\in P_{\chi}$ and $u,v\in S\backslash I_{\chi}$, then $x\in P_{\chi}$ and we have respectively $\rho(x)=\rho(p)\chi (u)$, $\rho(x)=\rho(p)\chi (v)$, or $\rho(x)=\rho(p)\chi (uv)$.\\
(II): $k(xy)=k(yx)=0$ for all $x\in S\backslash I_{\chi}$ and $y\in I_{\chi}\backslash P_{\chi}$.\par
Conversely, if $(k,l)$ are given by the formulas in (1), (2) or (3) with conditions (I) and (II) holding and if $k$ is abelian in case (3), then $(k,l)$ satisfies \eqref{E2}.\par
 Furthermore, if $S$ is a topological semigroup and $k\in C(S)$, then \\$\chi,\chi^*\in C(S)$, $A\in C(S\backslash I_{\chi})$, and $\rho \in C(P_{\chi})$.
\label{P1}
\end{prop}
\begin{proof}
We check by elementary computations that if $k$ and $l$ are of the forms (1)--(3) then $(k,l)$ is a solution of \eqref{E2}.\par
Conversely, we will discuss two cases according to whether $k$ and  $l$ are linearly dependent or not.\newline
\underline{First case :} $k$ and $l$ are linearly dependent. There exists $c\in \mathbb{C}$ such that $l=ck$. From \eqref{E2} we get 
\begin{equation}
\mu (y)k(x\sigma (y))=ck(x)k(y)-ck(y)k(x)=0,\quad\text{for all}\quad x,y \in S.
\label{E7}
\end{equation}
 Since $\mu(y\sigma(y))=1$ for all $y\in S$, we deduce from \eqref{E7} that $k=0$ on $S^2$. This occurs in (1) of Proposition \ref{P1}.\newline
\underline{Second case :} $k$ and $l$ are linearly independent. Since $\mu (y\sigma(y))=1$ for all $y\in S$, we multiply \eqref{E2} by $\mu(\sigma(y))$ and we get 
\begin{equation}
k(x\sigma(y))=\mu(\sigma(y))k(x)l(y)-\mu(\sigma(y))k(y)l(x),\quad \text{for all}\quad x,y\in S.
\label{D1}
\end{equation}
By using the associativity of the semigroup operation, we can compute $k(x\sigma(y)\sigma(z))$ first as $k((x\sigma(y))\sigma(z))$ and then as $k(x(\sigma(y)\sigma(z)))$ under the identity \eqref{D1} and compare the results to obtain 
\begin{multline}
-k(z)l(x\sigma(y))=k(x)\left[\mu(\sigma(y))l(yz)-\mu(\sigma(y))l(y)l(z) \right] +\\l(x)\left[\mu(\sigma(y))k(y)l(z)-\mu(\sigma(y))k(yz) \right]. 
\label{D2}
\end{multline}
Since $k\neq 0$, there exists $z_0\in S$ such that $k(z_0)\neq 0$, and then we deduce from the identity above that 
\begin{equation}
l(x\sigma(y))=k(x)f(y)+l(x)g(y),
\label{D3}
\end{equation}
where \\
$$f(y)=\dfrac{-1}{k(z_0)}\left[\mu(\sigma(y))l(yz_0)-\mu(\sigma(y))l(y)l(z_0) \right], $$
and 
\begin{equation}
g(y)=\dfrac{1}{k(z_0)}\left[\mu(\sigma(y))k(yz_0)-\mu(\sigma(y))k(y)l(z_0) \right]. 
\label{D4}
\end{equation}
By substituting \eqref{D3} into \eqref{D2}, we get
\begin{align*}
k(x)(-k(z)f(y))+l(x)(-k(z)g(y))=k(x)\left[\mu(\sigma(y))l(yz)- \mu(\sigma(y))l(y)l(z)\right]+\\l(x)\left[\mu(\sigma(y))k(y)l(z)-\mu(\sigma(y))k(yz) \right].  
\end{align*}
Using the linear independence of $k$ and $l$, we obtain 
$$\mu(\sigma(y))l(yz)=\mu(\sigma(y))l(y)l(z)-k(z)f(y),$$
and 
\begin{equation}
\mu(\sigma(y))k(yz)=\mu(\sigma(y))k(y)l(z)+k(z)g(y).
\label{D5}
\end{equation}
From \eqref{E2} and \eqref{D4} we deduce that the function $g(y)$ can be written as 
$$g(y)=\alpha \mu(\sigma(y))k(y)+\beta \mu(\sigma(y))l(y),$$
where 
$$\alpha =\dfrac{\mu(\sigma(z_0))l(\sigma(z_0))-l(z_0)}{k(z_0)}\quad \text{and}\quad \beta =\dfrac{-\mu(\sigma(z_0))k(\sigma(z_0))}{k(z_0)}.$$
From the last form of $g(y)$ and taking into account that $\mu (y\sigma(y))=1$ for all $y\in S$, equation \eqref{D5} can be written as follows
\begin{equation}
k(yz)=k(y)\left[l(z)+\alpha k(z) \right]+\beta k(z)l(y).
\label{Z1}
\end{equation} 
On the other hand from equation \eqref{D1} we get 
\begin{equation}
k(yz)=k(y)l^*(z)-k^*(z)l(y),
\label{D7}
\end{equation}
then by comparing \eqref{Z1} and \eqref{D7} and using the linear independence of $k$ and $l$ we obtain 
\begin{equation}
l(z)+\alpha k(z)=l^*(z),
\end{equation}
\begin{equation}
\beta k(z)=-k^*(z),\quad \text{for all}\quad z\in S.
\label{D6}
\end{equation}
Since $k\neq 0$, then we get from \eqref{D6} that $\beta \neq 0$ and $\beta^2 =1$. This means that $k=k^*$ or $k=-k^*$. If $k=k^*$ then according to Lemma \ref{lem} (2), $k(xyz)=0$ for all $x,y,z\in S$. This implies that \begin{equation}
k(x)l(yz)=k(yz)l(x),\quad \text{for all}\quad x,y,z\in S.
\label{err1}
\end{equation}  
Since $k$ and $l$ are linearly independent, then $k\neq 0$ on $S^2$, so there exists $y_0,z_0\in S$ such that $k(y_0z_0)\neq 0$. By letting $y=y_0$ and $z=z_0$ in \eqref{err1}, we obtain $l=b k$ for some constant $b\in \mathbb{C}$. This contradicts the fact that $k$ and $l$ are linearly independent. Now if $k=-k^*$ then by using the same computations as in the proof of \cite[Theorem 2.1]{BE}, we get $l^{\circ}=ck$  where $c\in \mathbb{C}$ is a constant and that the pair $(k,l^e)$ satisfies the sine addition law
\[k(xy)=k(x)l^e(y)+k(y)l^e(x),\quad x,y\in S.\]
Hence according to \cite[Theorem 3.1]{EB2} and taking into account that $k\neq 0$, the pair falls into two categories:\\
(i) $k=c_1\dfrac{\chi _1-\chi _2}{2}$ and $l^e=\dfrac{\chi _1+\chi _2}{2}$ , where $\chi _1, \chi _2 : S\rightarrow \mathbb{C}$ are different multiplicative functions and $c_1 \in \mathbb{C}\backslash \lbrace0\rbrace$ is a constant. Since $l=l^e+l^{\circ}$ and $k=-k^*$ we deduce that 
\[k=c_1\dfrac{\chi -\chi ^*}{2}\quad\text{and}\quad l=\dfrac{\chi +\chi ^*}{2}+c_2\dfrac{\chi -\chi ^*}{2},\]
where $\chi  : S\rightarrow \mathbb{C}$ is a multiplicative function such that $\chi ^*\neq \chi$ and $c_1 \in \mathbb{C}\backslash  \lbrace0\rbrace,c_2\in \mathbb{C}$ are constants. This occurs in part (2) of Proposition \ref{P1}.\\
(ii)$$k=\left\{ \begin{matrix}
   \chi A & on & S\backslash {{I}_{\chi }}  \\
   0 & on & {{I}_{\chi }}\backslash {{P}_{\chi }}  \\
   \rho  & on & {{P}_{\chi }}  \\
\end{matrix} \right.\quad\text{and}\quad l^e=\chi \ ,
$$ 
where $\chi: S \rightarrow \mathbb{C}$ is a non-zero multiplicative function and $A: S \backslash I_{\chi} \rightarrow \mathbb{C}$ is a non-zero additive function such that $\chi^{*}=\chi$, $A \circ \sigma=-A$, and $\rho: P_{\chi} \rightarrow \mathbb{C}$ is the restriction of $k$ to $P_\chi$ such that $\rho^*=-\rho$ and we have the conditions (I) and (II). Since $l=l^e+l^{\circ}$ and $l^{\circ}=ck$, we can see that
$$ l=\left\{ \begin{matrix}
   \chi (1+cA) & on & S\backslash {{I}_{\chi }}  \\
   0 & on & {{I}_{\chi }}\backslash {{P}_{\chi }}  \\
   c\rho  & on & {{P}_{\chi }}  \\
\end{matrix} \ .\right.
$$ 
This is part (3).\par 
The topological statements are easy to verify. This completes the proof of Proposition \ref{P1}.
\end{proof}
 Now we solve the functional equation \eqref{E1} on a t-compatible semigroup. A notion which was recently introduced by Ebanks \cite[Definition 4.2]{EBA}. 
 \begin{defn}
Let $S$ be a semigroup. We will say that $S$ is compatible, if $S^2=S$ and for every prime ideal $I\subset S$ the following condition holds.
\begin{equation}
\text{For each}\quad q\in I\quad\text{there exists}\quad w_q\in S\backslash I\quad\text{such that}\quad qw_q\in I^2.
\label{Z2}
\end{equation}
We say that a topological semigroup $S$ is t-compatible, if $S=S^2$ and condition \eqref{Z2} holds for every prime ideal $I$.
\end{defn}
The following result shows that the solutions of \eqref{E1} on a t-compatible semigroup  have the same forms found in \cite[Theorem 4.3]{Ajb} for  semigroups generated by their squares.
\begin{thm}
Let $S$ be a t-compatible semigroup, and suppose that  $f,g,h \in C(S)$ satisfy \eqref{E1}. Then $f,g,h$ belong to one of the four families below. In addition $\chi,\chi^*, \theta\in C(S)$, and $A\in C(S\backslash I_{\chi})$.
\begin{enumerate}
\item[(1)] $f=\theta$, $g=0$ and $h$ is arbitrary, where $\theta \in \mathcal{N}_{\mu}(\sigma, S)$.
\item[(2)] $f=\theta$, $g$ is arbitrary and $h=0$, where $\theta \in \mathcal{N}_{\mu}(\sigma, S)$.
\item[(3)] $f=\theta+\alpha \dfrac{\chi +\chi ^*}{2}+\beta \dfrac{\chi -\chi ^*}{2}$ and \newline $g\otimes h=2\left(\beta \dfrac{\chi +\chi ^*}{2}+\alpha \dfrac{\chi -\chi ^*}{2} \right)\otimes \dfrac{\chi -\chi ^*}{2} $, where $\alpha ,\beta \in \mathbb{C}$ are constants, $\theta \in\mathcal{N}_{\mu}(\sigma, S)$ and $\chi : S\rightarrow \mathbb{C}$ is a multiplicative function such that $\left( \alpha,\beta\right) \neq (0,0)$ and $\chi \neq \chi ^*$.
\item[(4)] $$
\left\{\begin{array}{l}
f=\theta +\dfrac{\alpha}{2}\chi A +\dfrac{\beta}{4}\chi A^2, \quad g=\alpha \chi+\beta \chi  A ,\quad h=\chi A,\quad \text { on } \quad S \backslash I_{\chi} \\
f=\theta, \quad g=0,\quad h=0 \quad \text { on } \quad I_\chi,
\end{array}\right.
$$
\end{enumerate} 
where $\alpha ,\beta \in \mathbb{C}$ are constants, $\theta \in\mathcal{N}_{\mu}(\sigma, S)$ and $\chi : S\rightarrow \mathbb{C}$ is a non-zero multiplicative function such that $\left( \alpha,\beta\right) \neq (0,0)$, $\chi = \chi ^*$ and $A:S\backslash I_{\chi} \rightarrow \mathbb{C}$ is a non-zero additive function such that $A\circ \sigma =-A$.

\label{T1}
\end{thm}
\begin{proof}
We check by elementary computations that if $f$, $g$ and $h$ are of the forms (1)--(4) then $(f,g,h)$ is a solution of \eqref{E1}.\newline
Let $f,g,h :S\rightarrow \mathbb{C}$ satisfy the functional equation \eqref{E1}. If $g=0$ then $h$ is arbitrary and $f=\theta$ where $\theta \in \mathcal{N}_{\mu}(\sigma, S)$. If $h=0$ then $g$ will be arbitrary and $f=\theta$ where $\theta \in \mathcal{N}_{\mu}(\sigma, S)$. So from now on we assume that $g\neq 0$ and $h\neq 0$. According to  \cite[Lemma 4.2 (1)]{Ajb}, there exists a function $l:S\rightarrow \mathbb{C}$ such that 
$$h(xy)=g^*(x)l(y)-g(y)l^*(x),\quad \text{for all}\quad x,y\in S.$$
This implies that $h(xy)=-h^*(yx)$, so $$h(xyz)=-h^*(zxy)=h(yzx)=-h^*(xyz),$$ for all $x,y,z \in S$. Since $S^2=S$ we deduce that $h=-h^*$ and then accordding to \cite[Lemma 4.2 (6)]{Ajb}, there exists a constant $b\in \mathbb{C}$ such that $g^{\circ}=bh$, so we discuss the two cases : $b=0$ and $b\neq 0$.\newline
\underline{First case :} $b\neq 0$. According to the proof of \cite[Theorem 4.3 case A]{Ajb}  there exists a function $m:S\rightarrow \mathbb{C}$ and a constant $c_3\in \mathbb{C}$ such that the pair $(h,m)$ satisfies the $\mu$-sine subtraction law \eqref{E2}, and $g^e=c_3m^e$, then according to Proposition \ref{P1} and taking into account that $h$ is a non-zero odd function, the pair $(h,m)$ falls into the categories:\\
(i) $h=c_1\dfrac{\chi -\chi ^*}{2}$ and $m=\dfrac{\chi +\chi ^*}{2}+\alpha \dfrac{\chi -\chi ^*}{2}$ where $c_1 \in \mathbb{C}\backslash \lbrace 0\rbrace, \alpha \in \mathbb{C}$ are constants and $\chi : S\rightarrow \mathbb{C}$ is a multiplicative function such that $\chi \neq \chi ^*$, then by using similar computations to the ones of the proof of \cite[Theorem 4.3 case A (i)]{Ajb}, we get that 
\[g=c_3\dfrac{\chi +\chi ^*}{2}+c_2\dfrac{\chi -\chi ^*}{2}\quad \text{and}\quad f=\theta +\dfrac{c_1}{2}\left(c_2\dfrac{\chi +\chi ^*}{2}+c_3\dfrac{\chi -\chi ^*}{2} \right), \]
where $c_2=bc_1 \in \mathbb{C}$ is a constant and $\theta \in \mathcal{N}_{\mu}(\sigma, S)$. This occurs in part (3) with $\alpha=\dfrac{c_1c_2}{2}$ and $\beta =\dfrac{c_1c_3}{2}$.\\
(ii) $$h=\left\{ \begin{matrix}
   \chi A & on & S\backslash {{I}_{\chi }}  \\
   0 & on & {{I}_{\chi }}\backslash {{P}_{\chi }}  \\
   \rho  & on & {{P}_{\chi }}  \\
\end{matrix} \right.\quad\text{and}\quad m=\left\{ \begin{matrix}
   \chi (1+cA) & on & S\backslash {{I}_{\chi }}  \\
   0 & on & {{I}_{\chi }}\backslash {{P}_{\chi }}  \\
   c\rho  & on & {{P}_{\chi }}  \\
\end{matrix} \ ,\right.
$$ 
where $c \in \mathbb{C}$ is a constant, $\chi: S \rightarrow \mathbb{C}$ is a non-zero multiplicative function and $A: S \backslash I_{\chi} \rightarrow \mathbb{C}$ is a non-zero additive function such that $\chi^{*}=\chi$, $A \circ \sigma=-A$,  and $\rho: P_{\chi} \rightarrow \mathbb{C}$ is a function such that $\rho^*=-\rho$ and $\rho$ satisfy the condition (I) from Proposition \ref{P1}.\newline
Let $y\in I_{\chi}$. Since $I_{\chi}$ is a prime ideal (if nonempty) and $S$ is t-compatible, then either $y\in I_{\chi}^2$ or there exists $w_y \in S\backslash I_{\chi}$ such that $yw_y \in I_{\chi}^2$. It follows  from condition (I) that $h(y)=0$. Proceeding exactly as in the proof of \cite[Theorem 4.3 case A (ii)]{Ajb}, we find that
\[g=c_3\chi +c_2 \chi A\quad\text{on}\quad S \backslash I_{\chi}\quad\text{and}\quad g=0\quad\text{on}\quad I_\chi,\]
and 
\[f=\theta +\dfrac{c_3}{2}\chi A+\dfrac{c_2}{4} \chi A^2\quad\text{on}\quad S \backslash I_{\chi}\quad\text{and}\quad f=\theta\quad\text{on}\quad I_\chi,\]
where $c, c_{2},c_3 \in \mathbb{C}$ are constants, $\chi: S \rightarrow \mathbb{C}$ is a non-zero multiplicative function and $A: S \backslash I_{\chi} \rightarrow \mathbb{C}$ is a non-zero additive function such that $\chi^{*}=\chi$, $A \circ \sigma=-A$ and $\theta \in \mathcal{N}_{\mu}(\sigma, S)$. This is part (4) of Theorem \ref{T1} with $\alpha=c_3$ and $\beta =c_2$.\\
\underline{Second case :} $b=0$. According to the proof of \cite[Theorem 4.3 case B]{Ajb} there exists  a constant $\lambda \in \mathbb{C}\backslash \lbrace 0\rbrace$ such that the pair $\left( h,\dfrac{\lambda}{2}g\right) $ satisfies the sine addition law  
\begin{equation}
h(xy)=h(x)\left( \dfrac{\lambda}{2}g(y)\right) +h(y)\left( \dfrac{\lambda}{2}g(x)\right), 
\end{equation}
for all $x,y \in S$. Then according to \cite[Theorem 3.1]{EB2} and taking into account that $h$ is a non-zero odd function  and $g$ is even, we have the following possibilities:\\
(i) $h=c_1\dfrac{\chi -\chi ^*}{2}$ and $g=c_2\dfrac{\chi +\chi ^*}{2}$, where $c_1,c_2 \in \mathbb{C}\backslash \lbrace 0\rbrace$ are constants and $\chi : S\rightarrow \mathbb{C}$ is a multiplicative function such that $\chi \neq \chi ^*$, then we get by  using the same computations as in the proof of  \cite[Theorem 4.3 case B (i)]{Ajb} that 
\[f=\theta +\dfrac{c_1c_2}{2}\dfrac{\chi -\chi ^*}{2},\]
where $\theta \in \mathcal{N}_{\mu}(\sigma, S)$. This is part (3) with $\alpha=0$ and $\beta =\dfrac{c_1c_2}{2}$.\\
(ii) $$h=\left\{ \begin{matrix}
   \chi A & on & S\backslash {{I}_{\chi }}  \\
   0 & on & {{I}_{\chi }}\backslash {{P}_{\chi }}  \\
   \rho  & on & {{P}_{\chi }}  \\
\end{matrix} \right.\quad\text{and}\quad g=c_1\chi \  ,
$$ 
where $c_1 \in \mathbb{C}\backslash \left\lbrace 0 \right\rbrace $ is a constant, $\chi: S \rightarrow \mathbb{C}$ is a non-zero multiplicative function and $A: S \backslash I_{\chi} \rightarrow \mathbb{C}$ is a non-zero additive function such that $\chi^{*}=\chi$, $A \circ \sigma=-A$, and $\rho: P_{\chi} \rightarrow \mathbb{C}$ is a function such that $\rho^*=-\rho$ and the condition (I) is holding. \newline
Proceeding exactly as in the previous case we show that $h=0$ on $I_{\chi}$. Then  by using similar computations to the ones of the proof of \cite[Theorem 4.3 case B (ii)]{Ajb}, we get that
\[f=\theta +\dfrac{c_1}{2}\chi A\quad\text{on}\quad S \backslash I_{\chi}, \]
where $\theta \in \mathcal{N}_{\mu}(\sigma, S)$. This occurs in part (4) with $\beta =0$ and $\alpha=\dfrac{c_1}{2}$. This completes the proof of Theorem \ref{T1}.
\end{proof}
\begin{rem}
Note that if the semigroup $S$ is not t-compatible, d'Alembert's equation \eqref{E1} may have nontrivial solutions taking arbitrary values at some points (See \cite[Example 5.3]{EB3}).
\end{rem}
\subsection*{Acknowledgment}

% ------------------------------------------------------------------------
\end{document}